\documentclass{alggeom}
\usepackage{amsmath}
\usepackage{amsxtra}
\usepackage{amscd}
\usepackage{amsthm}
\usepackage{amsfonts}
\usepackage{amssymb}
\usepackage{eucal}
\usepackage[matrix,arrow,curve]{xy}

\newtheorem{cor}[subsection]{Corollary}

\newtheorem{lem}[subsection]{Lemma}
\newtheorem{prop}[subsection]{Proposition}

\newtheorem{conj}[subsection]{Conjecture}
\newtheorem{thm}[subsection]{Theorem}

\newtheorem{rem}[subsection]{Remark}

\theoremstyle{definition}

\theoremstyle{remark}

\newcommand{\nc}{\newcommand}
\nc{\renc}{\renewcommand} \nc{\ssec}{\subsection}
\nc{\sssec}{\subsubsection} \nc{\on}{\operatorname}

\nc\ol{\overline} \nc\ul{\underline} \nc\wt{\widetilde}
\nc\tboxtimes{\wt{\boxtimes}} \nc{\alp}{\alpha}

\nc{\ZZ}{{\mathbb Z}} \nc{\NN}{{\mathbb N}} \nc{\CC}{{\mathbb C}}
\nc{\OO}{{\mathbb O}} \renc{\SS}{{\mathbb S}} \nc{\DD}{{\mathbb
D}}
\nc{\Fq}{{\mathbb F}_q} \nc{\Fqb}{\ol{{\mathbb F}_q}}
\nc{\Ql}{\ol{{\mathbb Q}_\ell}} \nc{\id}{\text{id}} \nc\X{\mathcal
X}

\nc{\Hom}{\on{Hom}} \nc{\Lie}{\on{Lie}} \nc{\Loc}{\on{Loc}}
\nc{\Pic}{\on{Pic}} \nc{\Bun}{\on{Bun}} \nc{\IC}{\on{IC}}
\nc{\Aut}{\on{Aut}} \nc{\rk}{\on{rk}} \nc{\Sh}{\on{Sh}}
\nc{\Perv}{\on{Perv}} \nc{\pos}{{\on{pos}}} \nc{\Conv}{\on{Conv}}
\nc{\Sph}{\on{Sph}} \nc{\Sym}{\on{Sym}}
\nc{\BunBb}{\overline{\Bun}_B} \nc{\Buno}{\overset{o}{\Bun}}
\nc{\BunPb}{{\overline{\Bun}_P}}
\nc{\BunBM}{\overline{\Bun}_{B(M)}}
\nc{\BunPbw}{{\widetilde{\Bun}_P}}
\nc{\BunBP}{\widetilde{\Bun}_{B,P}} \nc{\GUb}{\overline{G/U}}
\nc{\GUPb}{\overline{G/U(P)}}
\nc{\iso}{{\stackrel{\sim}{\longrightarrow}}}

\nc{\Hhom}{\underline{\on{Hom}}} \nc\syminfty{\on{Sym}^{\infty}}
\nc\lal{\ol{\lambda}} \nc\xl{\ol{x}} \nc\thl{\ol{\theta}}
\nc\nul{\ol{\nu}} \nc\mul{\ol{\mu}} \nc\Sum\Sigma
\nc{\oX}{\overset{o}{X}{}}

\nc{\M}{{\mathcal M}} \nc{\N}{{\mathcal N}} \nc{\F}{{\mathcal F}}
\nc{\D}{{\mathcal D}} \nc{\Q}{{\mathcal Q}} \nc{\Y}{{\mathcal Y}}
\nc{\G}{{\mathcal G}} \nc{\E}{{\mathcal E}} \nc{\CalC}{{\mathcal
C}}
\nc\Dh{\widehat{\D}}

\nc{\C}{{\mathcal C}} \nc{\K}{{\mathcal K}}
\renewcommand{\H}{{\mathcal H}}
\renewcommand{\S}{{\mathcal S}}
\nc{\T}{{\mathcal T}} \nc{\V}{{\mathcal V}} \renc{\P}{{\mathcal
P}} \nc{\A}{{\mathcal A}} \nc{\B}{{\mathcal B}} \nc{\U}{{\mathcal
U}}

\nc{\Gr}{\on{Gr}}

\nc{\frn}{{\check{\mathfrak u}(P)}}
\nc\f{{\mathfrak f}}

\nc{\q}{{\mathfrak q}} \nc{\p}{{\mathfrak p}} \nc{\s}{{\mathfrak
s}} \nc\w{\text{w}}

\nc\mathi\iota \nc\Spec{\on{Spec}} \nc\Mod{\on{Mod}}
\nc{\tw}{\widetilde{\mathfrak t}} \nc{\pw}{\widetilde{\mathfrak
p}} \nc{\qw}{\widetilde{\mathfrak q}} \nc{\jw}{\widetilde j}

\nc{\grb}{\overline{\Gr}} \nc{\I}{\mathcal I}

\nc{\lambdach}{{\check\lambda}} \nc{\Lambdach}{{\check\Lambda}{}}
\nc{\much}{{\check\mu}} \nc{\omegach}{{\check\omega}}
\nc{\nuch}{{\check\nu}} \nc{\etach}{{\check\eta}}
\nc{\alphach}{{\check\alpha}} \nc{\betach}{{\check\beta}}
\nc{\rhoch}{{\check\rho}} \nc{\ch}{{\check h}}

\nc{\Hb}{\overline{\H}}

\nc{\BA}{{\mathbb{A}}} \nc{\BC}{{\mathbb{C}}} \nc{\BQ}{{\mathbb{Q}}}
\nc{\BM}{{\mathbb{M}}} \nc{\BN}{{\mathbb{N}}} \nc{\BO}{{\mathbb{O}}}
\nc{\BP}{{\mathbb{P}}} \nc{\BR}{{\mathbb{R}}}
\nc{\BZ}{{\mathbb{Z}}} \nc{\BS}{{\mathbb{S}}}

\nc{\CA}{{\mathcal{A}}} \nc{\CB}{{\mathcal{B}}}
\nc{\CE}{{\mathcal{E}}} \nc{\CF}{{\mathcal{F}}}
\nc{\CG}{{\mathcal{G}}} \nc{\CH}{{\mathcal{H}}}
\nc{\CI}{{\mathcal{I}}} \nc{\CL}{{\mathcal{L}}}
\nc{\CM}{{\mathcal{M}}} \nc{\CN}{{\mathcal{N}}}
\nc{\CO}{{\mathcal{O}}} \nc{\CP}{{\mathcal{P}}}
\nc{\CQ}{{\mathcal{Q}}} \nc{\CR}{{\mathcal{R}}}
\nc{\CS}{{\mathcal{S}}} \nc{\CT}{{\mathcal{T}}}
\nc{\CU}{{\mathcal{U}}} \nc{\CV}{{\mathcal{V}}}  \nc{\CY}{{\mathcal Y}}
\nc{\CW}{{\mathcal{W}}} \nc{\CZ}{{\mathcal{Z}}}

\nc{\cM}{{\check{\mathcal M}}{}} \nc{\csM}{{\check{\mathcal A}}{}}
\nc{\oM}{{\overset{\circ}{\mathcal M}}{}}
\nc{\obM}{{\overset{\circ}{\mathbf M}}{}}
\nc{\oCA}{{\overset{\circ}{\mathcal A}}{}}
\nc{\obA}{{\overset{\circ}{\mathbf A}}{}}
\nc{\ooM}{{\overset{\circ}{M}}{}}
\nc{\osM}{{\overset{\circ}{\mathsf M}}{}}
\nc{\vM}{{\overset{\bullet}{\mathcal M}}{}}
\nc{\nM}{{\underset{\bullet}{\mathcal M}}{}}
\nc{\oD}{{\overset{\circ}{\mathcal D}}{}}
\nc{\obD}{{\overset{\circ}{\mathbf D}}{}}
\nc{\oA}{{\overset{\circ}{\mathbb A}}{}}
\nc{\op}{{\overset{\bullet}{\mathbf p}}{}}
\nc{\cp}{{\overset{\circ}{\mathbf p}}{}}
\nc{\oU}{{\overset{\bullet}{\mathcal U}}{}}
\nc{\oZ}{{\overset{\circ}{\mathcal Z}}{}}
\nc{\ofZ}{{\overset{\circ}{\mathfrak Z}}{}}

\nc{\ff}{{\mathfrak{f}}} \nc{\fv}{{\mathfrak{v}}}
\nc{\fa}{{\mathfrak{a}}} \nc{\fb}{{\mathfrak{b}}}
\nc{\fd}{{\mathfrak{d}}} \nc{\fe}{{\mathfrak{e}}}
\nc{\fg}{{\mathfrak{g}}} \nc{\fgl}{{\mathfrak{gl}}}
\nc{\fh}{{\mathfrak{h}}} \nc{\fri}{{\mathfrak{i}}}
\nc{\fj}{{\mathfrak{j}}} \nc{\fk}{{\mathfrak{k}}}
\nc{\fm}{{\mathfrak{m}}} \nc{\fn}{{\mathfrak{n}}}
\nc{\ft}{{\mathfrak{t}}} \nc{\fu}{{\mathfrak{u}}}
\nc{\fw}{{\mathfrak{w}}} \nc{\fz}{{\mathfrak{z}}}
\nc{\fp}{{\mathfrak{p}}} \nc{\frr}{{\mathfrak{r}}}
\nc{\fs}{{\mathfrak{s}}} \nc{\fo}{{\mathfrak{o}}}
\nc{\fsl}{{\mathfrak{sl}}} \nc{\fsp}{{\mathfrak{sp}}}
\nc{\hsl}{{\widehat{\mathfrak{sl}}}}
\nc{\hgl}{{\widehat{\mathfrak{gl}}}}
\nc{\hg}{{\widehat{\mathfrak{g}}}}
\nc{\chg}{{\widehat{\mathfrak{g}}}{}^\vee}
\nc{\hn}{{\widehat{\mathfrak{n}}}}
\nc{\chn}{{\widehat{\mathfrak{n}}}{}^\vee}

\nc{\fA}{{\mathfrak{A}}} \nc{\fB}{{\mathfrak{B}}}
\nc{\fD}{{\mathfrak{D}}} \nc{\fE}{{\mathfrak{E}}}
\nc{\fF}{{\mathfrak{F}}} \nc{\fG}{{\mathfrak{G}}} \nc{\fH}{{\mathfrak{H}}}
\nc{\fI}{{\mathfrak{I}}} \nc{\fJ}{{\mathfrak{J}}}
\nc{\fK}{{\mathfrak{K}}} \nc{\fL}{{\mathfrak{L}}}
\nc{\fM}{{\mathfrak{M}}} \nc{\fN}{{\mathfrak{N}}}
\nc{\frP}{{\mathfrak{P}}} \nc{\fQ}{{\mathfrak{Q}}}
\nc{\fT}{{\mathfrak{T}}} \nc{\fU}{{\mathfrak{U}}}
\nc{\fV}{{\mathfrak{V}}} \nc{\fW}{{\mathfrak{W}}}
\nc{\fX}{{\mathfrak{X}}} \nc{\fY}{{\mathfrak{Y}}}
\nc{\fZ}{{\mathfrak{Z}}}

\nc{\ba}{{\mathbf{a}}}
\nc{\bb}{{\mathbf{b}}} \nc{\bc}{{\mathbf{c}}}
\nc{\be}{{\mathbf{e}}} \nc{\bj}{{\mathbf{j}}}
\nc{\bn}{{\mathbf{n}}} \nc{\bp}{{\mathbf{p}}}
\nc{\bq}{{\mathbf{q}}} \nc{\br}{{\mathbf{r}}} \nc{\bt}{{\mathbf{t}}}
\nc{\bfu}{{\mathbf{u}}} \nc{\bv}{{\mathbf{v}}}
\nc{\bx}{{\mathbf{x}}} \nc{\by}{{\mathbf{y}}}
\nc{\bw}{{\mathbf{w}}} \nc{\bA}{{\mathbf{A}}}
\nc{\bB}{{\mathbf{B}}} \nc{\bC}{{\mathbf{C}}}
\nc{\bD}{{\mathbf{D}}} \nc{\bF}{{\mathbf{F}}}
\nc{\bH}{{\mathbf{H}}} \nc{\bK}{{\mathbf{K}}}
\nc{\bM}{{\mathbf{M}}} \nc{\bN}{{\mathbf{N}}}
\nc{\bO}{{\mathbf{O}}} \nc{\bS}{{\mathbf{S}}} \nc{\bT}{{\mathbf{T}}}
\nc{\bV}{{\mathbf{V}}} \nc{\bW}{{\mathbf{W}}}
\nc{\bX}{{\mathbf{X}}}
\nc{\bY}{{\mathbf{Y}}} \nc{\bP}{{\mathbf{P}}}
\nc{\bZ}{{\mathbf{Z}}} \nc{\bh}{{\mathbf{h}}}

\nc{\sA}{{\mathsf{A}}} \nc{\sB}{{\mathsf{B}}}
\nc{\sC}{{\mathsf{C}}} \nc{\sD}{{\mathsf{D}}}
\nc{\sE}{{\mathsf{E}}} \nc{\sF}{{\mathsf{F}}}
\nc{\sK}{{\mathsf{K}}} \nc{\sL}{{\mathsf{L}}}
\nc{\sM}{{\mathsf{M}}} \nc{\sO}{{\mathsf{O}}}
\nc{\sQ}{{\mathsf{Q}}} \nc{\sP}{{\mathsf{P}}}
\nc{\sT}{{\mathsf{T}}} \nc{\sZ}{{\mathsf{Z}}}
\nc{\sV}{{\mathsf{V}}}
\nc{\sfp}{{\mathsf{p}}} \nc{\sr}{{\mathsf{r}}}
\nc{\st}{{\mathsf{t}}} \nc{\sfb}{{\mathsf{b}}}
\nc{\sfc}{{\mathsf{c}}} \nc{\sd}{{\mathsf{d}}}
\nc{\sz}{{\mathsf{z}}}

\nc{\BK}{{\bar{K}}}

\nc{\tA}{{\widetilde{\mathbf{A}}}}
\nc{\tB}{{\widetilde{\mathcal{B}}}}
\nc{\tg}{{\widetilde{\mathfrak{g}}}} \nc{\tG}{{\widetilde{G}}}
\nc{\TM}{{\widetilde{\mathbb{M}}}{}}
\nc{\tO}{{\widetilde{\mathsf{O}}}{}}
\nc{\tU}{{\widetilde{\mathfrak{U}}}{}} \nc{\TZ}{{\tilde{Z}}}
\nc{\tx}{{\tilde{x}}} \nc{\tbv}{{\tilde{\bv}}}
\nc{\tfP}{{\widetilde{\mathfrak{P}}}{}} \nc{\tz}{{\tilde{\zeta}}}
\nc{\tmu}{{\tilde{\mu}}}

\nc{\urho}{\underline{\rho}} \nc{\uB}{\underline{B}}
\nc{\uC}{{\underline{\mathbb{C}}}} \nc{\ui}{\underline{i}}
\nc{\uj}{\underline{j}} \nc{\ofP}{{\overline{\mathfrak{P}}}}
\nc{\oB}{{\overline{\mathcal{B}}}}
\nc{\og}{{\overline{\mathfrak{g}}}} \nc{\oI}{{\overline{I}}}

\nc{\eps}{\varepsilon} \nc{\hrho}{{\hat{\rho}}}
\nc{\blambda}{{\boldsymbol{\lambda}}}

\nc{\one}{{\mathbf{1}}} \nc{\two}{{\mathbf{t}}}

\nc{\Rep}{{\mathop{\operatorname{\rm Rep}}}}
\nc{\Tot}{{\mathop{\operatorname{\rm Tot}}}}
\nc{\Ker}{{\mathop{\operatorname{\rm Ker}}}}
\nc{\Hilb}{{\mathop{\operatorname{\rm Hilb}}}}
\nc{\End}{{\mathop{\operatorname{\rm End}}}}
\nc{\Ext}{{\mathop{\operatorname{\rm Ext}}}}
\nc{\CHom}{{\mathop{\operatorname{{\mathcal{H}}\it om}}}}
\nc{\GL}{{\mathop{\operatorname{\rm GL}}}}
\nc{\gr}{{\mathop{\operatorname{\rm gr}}}}
\nc{\Id}{{\mathop{\operatorname{\rm Id}}}}
\nc{\defi}{{\mathop{\operatorname{\rm def}}}}
\nc{\length}{{\mathop{\operatorname{\rm length}}}}
\nc{\supp}{{\mathop{\operatorname{\rm supp}}}}

\nc{\Cliff}{{\mathsf{Cliff}}}
\nc{\Fl}{{\mathsf{Fl}}} \nc{\Fib}{{\mathsf{Fib}}}
\nc{\Coh}{{\mathsf{Coh}}} \nc{\FCoh}{{\mathsf{FCoh}}}

\nc{\reg}{{\text{\rm reg}}}

\nc{\cplus}{{\mathbf{C}_+}} \nc{\cminus}{{\mathbf{C}_-}}
\nc{\cthree}{{\mathbf{C}_*}} \nc{\Qbar}{{\bar{Q}}}

\nc{\bOmega}{{\overline{\Omega}}}

\nc{\seq}[1]{\stackrel{#1}{\sim}}

\nc{\aff}{\operatorname{aff}}

\newcommand{\YO}{{\mathcal{Y}}}
\newcommand{\DO}{{\mathcal{D}}}
 \DeclareMathOperator{\diag}{diag}
 \DeclareMathOperator{\Tr}{Tr}

\begin{document}

%\author{Michael Finkelberg and Leonid Rybnikov}
\title
%[Dva Idiota]
{Quantization of Drinfeld Zastava in type C}

\dedication{To Vladimir Drinfeld on his 60th birthday}

%\thanks{{\bf Mathematics Subject Classification (2010):}
%19E08, (22E65, 37K10).}

%\thanks{{\bf Key words:} Chainsaw quiver, Quadratic spaces, Affine Yangian,
%Hamiltonian reduction, Quantization.}

\author{Michael Finkelberg}
\email{fnklberg@gmail.com}
\address{IMU, IITP, and
National Research University Higher School of Economics\\
Department of Mathematics\\
20 Myasnitskaya st,
Moscow 101000, Russia}

\author{Leonid Rybnikov}
\email{leo.rybnikov@gmail.com}
\address{IITP, and
National Research University Higher School of Economics\\
Department of Mathematics\\
20 Myasnitskaya st,
Moscow 101000, Russia}
%\address{{\it Address}:\newline
%M.F.: IMU, IITP, and
%National Research University Higher School of Economics, \newline
%Department of Mathematics,\newline
%20 Myasnitskaya st,
%Moscow 101000, Russia \newline
%L.R.: IITP, and
%National Research University Higher School of Economics, \newline
%Department of Mathematics,\newline
%20 Myasnitskaya st,
%Moscow 101000, Russia}

\thanks{Both authors were partially supported by the RFBR grants 12-01-00944,
12-01-33101, 13-01-12401/13,
the National Research University Higher School of Economics' Academic Fund
award No.12-09-0062 and
the AG Laboratory HSE, RF government grant, ag. 11.G34.31.0023.
This study comprises research findings from the ``Representation Theory
in Geometry and in Mathematical Physics" carried out within The
National Research University Higher School of Economics' Academic Fund Program
in 2012, grant No 12-05-0014. L.R. was supported by the RFBR grant
11-01-93105-CNRSL-a. This study was carried out within The National
Research University Higher School of Economics Academic Fund Program
in 2013-2014, research grant No. 12-01-0065.}
%\email{\newline fnklberg@gmail.com, leo.rybnikov@gmail.com}

\begin{abstract}
Drinfeld zastava is a certain closure of the moduli space of maps from the
projective line to the Kashiwara flag scheme of an affine Lie algebra
$\hat \fg$. In case $\fg$ is the symplectic Lie algebra $\fsp_N$,
we introduce an affine, reduced, irreducible, normal quiver
variety $Z$ which maps to the zastava space isomorphically in characteristic 0.
The natural Poisson structure on the zastava space $Z$ can be
described in terms of Hamiltonian reduction of a certain Poisson
subvariety of the dual space of a (nonsemisimple) Lie algebra. The quantum
Hamiltonian reduction of the corresponding quotient of its universal enveloping
algebra produces a quantization $Y$ of the coordinate ring of $Z$. The same
quantization was obtained in the finite (as opposed to the affine) case
generically in the work of Gerasimov-Kharchev-Lebedev-Oblezin (2005).
%~\cite{o}.
We prove that $Y$ is a quotient of the affine
Borel Yangian. The analogous results for $\fg=\fsl_N$ were obtained in our
previous work. %\cite{fr}.
\end{abstract}
\maketitle

\section{Introduction}
\subsection{}
\label{111}
This note is a continuation of~\cite{fr} where we have studied the
Drinfeld zastava spaces $Z^{\ul{d}}(\widehat{\mathfrak{sl}}_N)$ from the
Invariant Theory viewpoint. Recall that given a collection of complex
vector spaces $(V_l)_{l\in\BZ/N\BZ},\ \ul\dim(V_l)_{l\in\BZ/N\BZ}=
(d_l)_{l\in\BZ/N\BZ}=\ul{d}$ along with $(W_l)_{l\in\BZ/N\BZ},\
\ul\dim(W_l)_{l\in\BZ/N\BZ}=(1,\ldots,1)$, we consider the space
$M_{\ul{d}}=\{(A_l,B_l,p_l,q_l)_{l\in\BZ/N\BZ}\}=$
$$\bigoplus_{l\in\BZ/N\BZ}\End(V_l)\oplus
\bigoplus_{l\in\BZ/N\BZ}\Hom(V_l,V_{l+1})\oplus
\bigoplus_{l\in\BZ/N\BZ}\Hom(W_{l-1},V_l)\oplus
\bigoplus_{l\in\BZ/N\BZ}\Hom(V_l,W_l)$$
of representations of the following chainsaw quiver:
$$
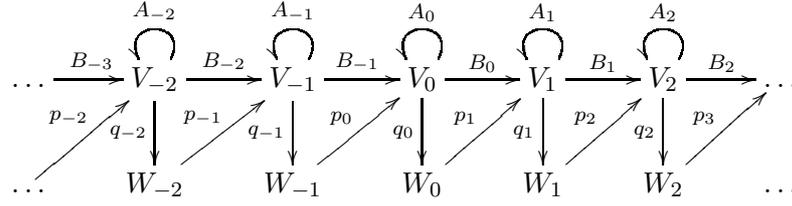
\xymatrix{
\ldots \ar[r]^{B_{-3}}
& V_{-2} \ar@(ur,ul)[]_{A_{-2}} \ar[r]^{B_{-2}} \ar[d]_{q_{-2}}
& V_{-1} \ar@(ur,ul)[]_{A_{-1}} \ar[r]^{B_{-1}} \ar[d]_{q_{-1}}
& V_0 \ar@(ur,ul)[]_{A_0} \ar[r]^{B_0} \ar[d]_{q_0}
& V_1 \ar@(ur,ul)[]_{A_1} \ar[r]^{B_1} \ar[d]_{q_1}
& V_2 \ar@(ur,ul)[]_{A_2} \ar[r]^{B_2} \ar[d]_{q_2} &\ldots\\
\ldots \ar[ur]^{p_{-2}} & W_{-2} \ar[ur]^{p_{-1}} & W_{-1} \ar[ur]^{p_0}
& W_0 \ar[ur]^{p_1} & W_1 \ar[ur]^{p_2} & W_2 \ar[ur]^{p_3} &\ldots
}$$
Furthermore, we consider the closed subscheme $\sM_{\ul{d}}\subset M_{\ul{d}}$
cut out by the equations $A_{l+1}B_l-B_lA_l+p_{l+1}q_l=0\ \forall l$, and two
open subschemes $M^s_{\ul{d}}\subset M_{\ul{d}}$
(resp. $M^c_{\ul{d}}\subset M_{\ul{d}}$) formed by all
$\{(A_l,B_l,p_l,q_l)_{l\in\BZ/N\BZ}\}$ such that for any $\BZ/N\BZ$-graded
subspace $V'_\bullet\subset V_\bullet$ with $A_lV'_l\subset V'_l$, and
$B_lV'_l\subset V'_{l+1}\ \forall l$, if $p_l(W_{l-1})\subset V'_l\ \forall l$,
then $V'_\bullet=V_\bullet$ (resp. if $V'_l\subset\on{Ker}q_l\ \forall l$, then
$V'_\bullet=0$). The group $G(V_\bullet)=\prod_{l\in\BZ/N\BZ}GL(V_l)$ acts naturally
on $M_{\ul{d}}$ preserving the subschemes $\sM_{\ul{d}},M^s_{\ul{d}},M^c_{\ul{d}}$.

According to~\cite{fr},~\cite{bf}, the action of $G(V_\bullet)$ on
$\sM_{\ul{d}}\cap M^s_{\ul{d}}\cap M^c_{\ul{d}}$ is free, and the quotient
$(\sM_{\ul{d}}\cap M^s_{\ul{d}}\cap M^c_{\ul{d}})/G(V_\bullet)$ is naturally
isomorphic to the moduli space $\overset{\circ}{Z}{}^{\ul{d}}$ of based maps
of degree $\ul{d}$ from the projective line to the Kashiwara flag scheme
of the affine Lie algebra $\widehat{\mathfrak{sl}}_N$. Moreover, the categorical
quotient $\sM_{\ul{d}}/\!/G(V_\bullet)$ is naturally isomorphic to the Drinfeld
zastava closure
$Z^{\ul{d}}(\widehat{\mathfrak{sl}}_N)\supset\overset{\circ}{Z}{}^{\ul{d}}$.
Furthermore, the scheme $\sM_{\ul{d}}/\!/G(V_\bullet)\simeq
Z^{\ul{d}}(\widehat{\mathfrak{sl}}_N)$ is reduced, irreducible, normal.
One of the crucial points in proving this consists in checking that
$(\sM_{\ul{d}}\cap M^s_{\ul{d}}\cap M^c_{\ul{d}})\subset\sM_{\ul{d}}$ is dense,
while $\sM_{\ul{d}}\subset M_{\ul{d}}$ is a complete intersection.

\subsection{}
\label{112}
One of the goals of this note is to extend the above results to the case
of zastava spaces for the affine {\em symplectic} Lie algebra
$\widehat{\mathfrak{sp}}_N$ (in case $N$ is even). We prove that
any zastava scheme for $\widehat{\mathfrak{sp}}_N$ is reduced, irreducible,
normal (note that these properties of zastava schemes were established
in~\cite{bf} for all {\em finite dimensional} simple Lie algebras).

To this end we again invoke Invariant Theory. Following~\cite{bs}, we equip
$W:=\bigoplus_{l\in\BZ/N\BZ}W_l$ with a symplectic form such that $W_l$ and
$W_k$ are orthogonal unless $l+k=N-1$. We equip $V:=\bigoplus_{l\in\BZ/N\BZ}V_l$
with a nondegenerate {\em symmetric} bilinear form such that
$V_l$ and $V_k$ are orthogonal unless $l+k=0$. In particular, we must have
$d_{-l}=d_l\ \forall l$, so that the collection $\ul{d}$ is encoded by
$\ol{d}:= (d_0,d_1,\ldots,d_{N/2})$. We denote by $O(V_\bullet)$
the Levi subgroup of the orthogonal group $O(V)$ preserving the
decomposition $V:=\bigoplus_{l\in\BZ/N\BZ}V_l$.
We consider the space $M^{-1}_{\ul{d}}\subset M_{\ul{d}}$ of representations of
the {\em quadratic chainsaw quiver} formed by all the selfadjoint collections
$A^*_l=A_l,\ B^*_l=B_{-l-1},\ p^*_l=q_{-l}$. We denote by $\sM^{-1}_{\ul{d}}\subset
M^{-1}_{\ul{d}}$ the scheme-theoretic intersection $\sM_{\ul{d}}\cap M^{-1}_{\ul{d}}$.

We prove that $\sM^{-1}_{\ul{d}}\cap M^s_{\ul{d}}=
\sM^{-1}_{\ul{d}}\cap M^c_{\ul{d}}\subset\sM^{-1}_{\ul{d}}$ is dense, while
$\sM^{-1}_{\ul{d}}\subset M^{-1}_{\ul{d}}$ is a complete intersection.
We deduce that the categorical quotient $\sM^{-1}_{\ul{d}}/\!/O(V_\bullet)$
is reduced, irreducible and normal. Furthermore, we prove that the
action of $O(V_\bullet)$ on $\sM^{-1}_{\ul{d}}\cap M^s_{\ul{d}}=
\sM^{-1}_{\ul{d}}\cap M^c_{\ul{d}}$ is free, and the quotient
$(\sM^{-1}_{\ul{d}}\cap M^s_{\ul{d}})/O(V_\bullet)$ is naturally isomorphic to
the moduli space $\overset{\circ}{Z}{}^{\ol{d}}(\widehat{\mathfrak{sp}}_N)$
of based maps
of degree $\ol{d}$ from the projective line to the Kashiwara flag scheme
of the affine Lie algebra $\widehat{\mathfrak{sp}}_N$. Moreover, the categorical
quotient $\sM^{-1}_{\ul{d}}/\!/G(V_\bullet)$ is naturally isomorphic to the Drinfeld
zastava closure $Z^{\ol{d}}(\widehat{\mathfrak{sp}}_N)\supset
\overset{\circ}{Z}{}^{\ol{d}}(\widehat{\mathfrak{sp}}_N)$.

\subsection{}
\label{113}
Quite naturally, we would like to extend the above results to the case of the
affine orthogonal Lie algebra $\widehat{\mathfrak{so}}_N$. To this end we
change the parities of the bilinear forms in~\ref{112}. That is, we equip
$W$ with a nondegenerate {\em symmetric} bilinear form, and we equip
$V$ with a {\em symplectic} form. The corresponding space of representations of
the quadratic chainsaw quiver is denoted by $M^1_{\ul{d}}$, and the
corresponding Levi subgroup of $Sp(V)$ is denoted by $Sp(V_\bullet)$.
It is still true that the action of $Sp(V_\bullet)$ on
$\sM^1_{\ul{d}}\cap M^s_{\ul{d}}=\sM^1_{\ul{d}}\cap M^c_{\ul{d}}$ is free,
and the quotient $(\sM^1_{\ul{d}}\cap M^s_{\ul{d}})/Sp(V_\bullet)$ is naturally
isomorphic to $\overset{\circ}{Z}{}^{\ol{d}}(\widehat{\mathfrak{so}}_N)$.

However, we encounter the following mysterious obstacle:
$\sM^1_{\ul{d}}$ is {\em not} irreducible in general, and
$\sM^1_{\ul{d}}\cap M^s_{\ul{d}}=\sM^1_{\ul{d}}\cap M^c_{\ul{d}}$ is only dense in
one of its irreducible components. For this reason the categorical quotient
$\sM^1_{\ul{d}}/\!/Sp(V_\bullet)$ is {\em not} isomorphic to the zastava space
$Z^{\ol{d}}(\widehat{\mathfrak{so}}_N)$ in general. The simplest example
occurs when $\ol{d}$ is the affine simple coroot of
$\widehat{\mathfrak{so}}_N$, that is $\ul{d}=(\ldots,0,1,2,1,0,\ldots)$.

\subsection{}%{Poisson structure and quantization}
\label{114}
Following \cite{fr}, we describe the natural Poisson structure on $Z^{\ol{d}}(\widehat{\mathfrak{sp}}_N)$ in quiver terms. It is obtained by the Hamiltonian reduction of a Poisson subvariety of the dual vector space of a (nonsemisimple) Lie
algebra $\fa_{\ul{d}}^{-1}$ with its Lie-Kirillov-Kostant bracket. Now the ring of functions $\BC[Z^{\ol{d}}(\widehat{\mathfrak{sp}}_N)]$ admits a natural quantization $\CY_{\ol{d}}^{-1}$ as the quantum Hamiltonian reduction of a quotient
algebra of the universal enveloping algebra $U(\fa_{\ul{d}}^{-1})$. The algebra
$\CY_{\ol{d}}^{-1}$ admits a homomorphism from the Borel subalgebra $\YO^{-1}$
of the Yangian of type $C$ in the case of finite Zastava space. We prove
that this homomorphism is surjective. In the affine situation, there is an affine analog $\widehat{\YO}{}^{-1}$ of $\YO^{-1}$ (it is no longer a subalgebra in the Yangian of $\widehat{\mathfrak{sp}}_N$), and we define it explicitly by generators and relations. We prove that there is a surjective homomorphism $\widehat{\YO}{}^{-1}\to\CY_{\ol{d}}^{-1}$. Moreover, we write down certain elements in the kernel of this homomorphism and conjecture that they generate the kernel (as a two-sided
ideal). These elements are similar to the generators of the kernel of the Kamnitzer-Webster-Weekes-Yacobi homomorphism from {\em shifted Yangian} to the quantization of the transversal slices in the affine Grassmannian. In fact, as explained in \cite{KWWY} $\CY_{\ol{d}}^{-1}$ as a filtered algebra is the limit of a sequence of quantum coordinate rings of transversal slices.

%\subsection{Acknowledgments} We are grateful to D.~Panyushev and D.~Timashev
%for the help with references.
%Both authors were partially supported by the RFBR grants 12-01-00944,
%12-01-33101, 13-01-12401/13,
%the National Research University Higher School of Economics' Academic Fund
%award No.12-09-0062 and
%the AG Laboratory HSE, RF government grant, ag. 11.G34.31.0023.
%This study comprises research findings from the ``Representation Theory
%in Geometry and in Mathematical Physics" carried out within The
%National Research University Higher School of Economics' Academic Fund Program
%in 2012, grant No 12-05-0014. L.R. was supported by the RFBR grant
%11-01-93105-CNRSL-a. This study was carried out within The National
%Research University Higher School of Economics Academic Fund Program
%in 2013-2014, research grant No. 12-01-0065.

\section{A quiver approach to Drinfeld Zastava for symplectic groups}

\subsection{Quadratic spaces}
\label{quadratic}
We will recall the convenient terminology introduced in~\cite{kp}.
Let $U$ be an $N$-dimensional complex vector space equipped with a
nondegenerate bilinear form $(,)$ such that $(u,v)=\varepsilon(v,u)$.
It will be called a {\em quadratic space of type $\varepsilon$} (shortly
an {\em orthogonal space} in case $\varepsilon=1$, a {\em symplectic space}
in case $\varepsilon=-1$). We denote by $G_\varepsilon(U)$ the subgroup of
$GL(U)$ leaving the form invariant. So we have $G_\varepsilon(U)=O(N)$
or $Sp(N)$ according to $\varepsilon=1$ or $\varepsilon=-1$.
%We denote by
%$G=\widetilde{G}{}^\circ_\varepsilon(W)$ the universal cover of the neutral
%component of $G_\varepsilon(W)$. So we have $G=Spin(N)$ or $Sp(N)$.

Let $A\mapsto A^*,\ \End(U)\to\End(U)$ be the canonical involution associated
to the form, i.e. $(Au,v)=(u,A^*v)$ for any $u,v\in U$. More generally, for
a linear operator $B\in\Hom(U,\ 'U)$ we denote by $B^*$ the adjoint
(or transposed) operator $B^*\in\Hom(\ 'U^*,U^*)$.

We choose a basis $w_0,\ldots,w_{N-1}$ in a quadratic space $W$ of type
$\varepsilon=-1$ such that for $0\leq l<N/2$
we have $(w_l,w_m)=\delta_{m,N-1-l}$ (note that $N$ is necessarily even).
The linear span of $w_l$ will be denoted
by $W_l\cong\BC$. We will often parametrize the base vectors
by the elements of $\BZ/N\BZ$. We define $I:=\{0,1,\ldots, N/2\}\subset
\{0,\ldots,N-1\}=\BZ/N\BZ$.
We set $I=I_0\sqcup I_1$ where $I_0=\{1,\ldots,\frac{N}{2}-1\},\
I_1=\{0,\frac{N}{2}\}$.

We choose another quadratic space $V$ of type $-\varepsilon=1$ decomposed into
direct sum $V=\bigoplus_{l\in\BZ/N\BZ}V_l$ such that $V_l$ is orthogonal to
$V_m$ unless $l+m=0\in\BZ/N\BZ$. Let $d_l$ denote the dimension of $V_l$.
We set $\ul{d}:=(d_l)_{l\in I}$.
%In case when the restriction of $(,)$ to $V_l$ is a (nondegenerate) symplectic
%form, we set $d_l:=\sd_l/2$. Otherwise we set $d_l:=\sd_l$,
%and $\ul{d}:=(d_l)_{l\in I}$.

We denote by $G_{-\varepsilon}(V_\bullet)$ the Levi subgroup of
$G_{-\varepsilon}(V)$ preserving the decomposition
$V=\bigoplus_{l\in\BZ/N\BZ}V_l$. It is isomorphic to
$O(V_0)\times O(V_{N/2})\times
\prod_{0<l<N/2}GL(V_l)$.
%in case $N$ is even, and to
%$G_{-\varepsilon}(V_0)\times\prod_{0<l<N/2}GL(V_l)$ in case $N$ is odd.

\subsection{Quadratic Chainsaw Quivers}
\label{QCQ}

Following~\cite[2.3]{fr} we consider the affine space $M^\varepsilon_{\ul{d}}$
of collections $(A_l,B_l,p_l,q_l)_{l\in\BZ/N\BZ}$ where
$A_l\in\End(V_l),\ B_l\in\Hom(V_l,V_{l+1}),\ p_l\in\Hom(W_{l-1},V_l),\
q_l\in\Hom(V_l,W_l)$ satisfy the following selfadjointness conditions:
$A_l^*=A_{-l},\ B_l^*=B_{-l-1},\ p_l^*=q_{-l}$. Here we view $p_l$ (resp. $q_l$)
as a vector (resp. covector) of $V_l$ using the identification of all $W_m$
with $\BC$.

Following {\em loc. cit.} we consider the subscheme
$\sM^\varepsilon_{\ul{d}}\subset M^\varepsilon_{\ul{d}}$ parametrizing the
$\ul{d}$-dimensional representations of the Chainsaw Quiver with bilinear form
(or the {\em Quadratic Chainsaw Quiver} for short), cut out by the equations
$A_{l+1}B_l-B_lA_l+p_{l+1}q_l=0\ \forall l$.

Clearly, $\sM^\varepsilon_{\ul{d}}$ is acted upon by the Levi subgroup
$G_{-\varepsilon}(V_\bullet)$, and we denote by $\fZ^\varepsilon_{\ul{d}}$
the categorical quotient $\sM^\varepsilon_{\ul{d}}/\!/G_{-\varepsilon}(V_\bullet)$.

{\bf Assumption.} From now on we set $\varepsilon=-1$.

\subsection{Examples}
\label{examples}
We consider three basic examples in types $C_1,C_2,\widetilde{C}_1$.

\subsubsection{$C_1$}
\label{c1}
We take $N\geq2,\ d_0=\ldots=d_{N/2-1}=0,\ d_{N/2}=d$. We have $V_{N/2}=V=\BC^d,\
A_{N/2}=A=A^*\in\End(V),\ B_1=0,\ p_1=p\in V,\ q_1=q\in V^*,\ q(v)=(p,v)$.
Thus $\sM^\varepsilon_{\ul{d}}=\End^+(V)\oplus V$, and
$\fZ^\varepsilon_{\ul{d}}=(\End^+(V)\oplus V)//O(V)$ where
$\End^+(V)\subset\End(V)$ stands for the linear subspace of selfadjoint
operators (symmetric matrices). By the classical Invariant Theory, the ring
of $O(V)$-invariant functions on $\End^+(V)\oplus V$ is freely generated by
the functions $a_1,\ldots,a_d,b_0,\ldots,b_{d-1}$ where $a_m:=\on{Tr}(A^m)$,
and $b_m=(p,A^mp)$. Hence $\fZ_{\ul{d}}^\varepsilon\simeq\BA^{2d}$.

\subsubsection{$C_2$}
\label{c2}
We take $N=4,\ d_0=0,\ d_1=d_2=d_3=1$. We have $V_1=V_2=V_3=\BC$, and hence
all our linear operators act between one-dimensional vector spaces, and
can be written just as numbers. Hence $M^\varepsilon_{\ul{d}}$ has coordinates
$A_1=A_3,A_2,B_1=B_2,q_2=-p_2,q_1=p_3,q_3=p_1$, and $\sM^\varepsilon_{\ul{d}}$
is cut out by a single equation $B_1(A_1-A_2)=p_2p_3$.
The group $G_{-\varepsilon}(V_\bullet)$ is the product
$GL(V_1)\otimes O(V_2)\simeq \BC^*\times\{\pm1\}$ with coordinates
$c\in\BC^*,\ s=\pm1$. It acts on $M^\varepsilon_{\ul{d}}$ as follows:
$(c,s)\cdot(A_1,A_2,B_1,p_1,p_2,p_3)=(A_1,A_2,csB_1,c^{-1}p_1,sp_2,cp_3)$.
%The loop rotation acts: $q\cdot(A_1,A_2,B_1,p_1,p_2,p_3)=
%(qA_1,qA_2,B_1,qp_1,qp_2,p_3)$.
The ring of $\BC^*\times\{\pm1\}$-invariant functions on
$\sM^\varepsilon_{\ul{d}}$ is generated by the functions $A_1,\ A_2,\
b_{12}:=p_2^2,\ b_{01}:=p_1p_3,\ b_{02}:=p_2B_1p_1,\ b_{03}:=B_1^2p_1^2$ with
three quadratic relations: $b_{02}(A_1-A_2)=b_{01}b_{12},\
b_{03}(A_1-A_2)=b_{01}b_{02},\ b_{02}^2=b_{12}b_{03}$. Thus
$\fZ_{\ul{d}}^\varepsilon$ is
a 4-dimensional (noncomplete) intersection of 3 quadrics in $\BA^6$.
According to~\cite{v}, $\fZ_{\ul{d}}^\varepsilon$ is reduced, not $\BQ$-Gorenstein,
but Cohen-Macaulay, normal, and has rational singularities.

\subsubsection{$\widetilde{C}_1$}
\label{tc1}
We take $N=2,\ d_0=d_1=1$. We have $V_1=\BC=V_2$, and hence all our linear
operators act between one-dimensional vector spaces, and
can be written just as numbers. Hence $M^\varepsilon_{\ul{d}}$ has coordinates
$A_1,A_2,B_0=B_1,q_0=p_0,q_1=-p_1$, and $\sM^\varepsilon_{\ul{d}}$
is cut out by a single equation $B_0(A_1-A_0)+p_1q_0=0$.
The group $G_{-\varepsilon}(V_\bullet)$ is the product
$O(V_0)\times O(V_1)\simeq\{\pm1\}\times\{\pm1\}$ with coordinates $(s_1,s_2)$.
The ring of $\{\pm1\}\times\{\pm1\}$-invariant functions on
$\sM^\varepsilon_{\ul{d}}$ is generated by the functions $A_0,\ A_1,\
b_1:=p_1^2,\ b_0:=p_0^2,\ s:=B_0^2$ with a single relation
$b_1b_0-s(A_0-A_1)^2=0$. Note the coincidence with the output
of~\cite[Example~2.8.3]{fr}.

\subsection{Dimension of $\sM^\varepsilon_{\ul{d}}$}
\label{dim M}
%From now on we set $\varepsilon=-1$.
We define the {\em factorization morphism} $\Upsilon:\ \sM^\varepsilon_{\ul{d}}
\to\BA^{\ul{d}}=\prod_{l\in I}(\BA^{(1)})^{(d_l)}$ so that the component
$\Upsilon_l$ is just $\on{Spec}A_l$.
%In case
%either $l\in I_0$ or $\varepsilon=-1$, the component $\Upsilon_l$ is just
%$\on{Spec}A_l$. In case $l\in I_1$ and $\varepsilon=1$, the endomorphism $A_l$
%must be selfadjoint with respect to a symplectic form on $V_l$. This implies
%that the traces of odd external powers of $A_l$ all vanish, so that the
%characteristic polynomial $\on{Char}(A_l;t)$ of $A_l$ contains only even powers
%of the variable $t$. Thus we can view $\on{Char}(A_l;t)$ as a monic polynomial
%$\on{Char}(A_l;u)$ of degree $d_l=\sd_l/2$ in the variable $u=t^2$. We define
%the $l$-th component $\Upsilon_l$ as this polynomial.

\begin{prop}
\label{fibers}
Every fiber of $\Upsilon$ has dimension $\dim G_{-\varepsilon}(V_\bullet)+
\sum_{l\in I}d_l$.
\end{prop}

\proof The same argument as in the proof of~\cite[Proposition~2.11]{fr}.
We just list the minor changes necessary in the quadratic case. The dimension
estimate for a general fiber is reduced to the the dimension estimate for the
zero fiber where all $A_l$ are nilpotent. By the adjointness condition,
$(A_\bullet,B_\bullet,p_\bullet,q_\bullet)$ is determined by its components
$A_l,\ 0\leq l\leq N/2;\ B_l,\ 0\leq l<N/2;\ p_l,\ 0<l\leq N/2;\ q_l,\
0\leq l<N/2$, and $A_0\in\End^+(V_0),\ A_{N/2}\in\End^+(V_{N/2})$.
Note that $\dim\End^+(V_l)=\frac{d_l(d_l+1)}{2},\ l=0,N/2$, while
$\dim O(V_l)=\frac{d_l(d_l-1)}{2},\ l=0,N/2$. The dimension of the space of
nilpotent selfadjoint operators in $\End^+(V_l)$ equals
$\frac{d_l(d_l-1)}{2},\ l=0,N/2$.
More generally, the space $\BO_\lambda^+\subset\End^+(V_l)$
of nilpotent selfadjoint operators of Jordan type $\lambda$ (a partition of
$d_l$) is a finite union of $O(V_l)$-orbits all of the same dimension
$\dim\BO_\lambda^+=\frac{1}{2}\dim\BO_\lambda$ (where $\BO_\lambda$ is the nilpotent
$GL(V_l)$-orbit consisting of nilpotent matrices of Jordan type $\lambda$),
see~\cite[Proposition~5]{kr}.
The argument of the proof of~\cite[Proposition~2.11]{fr} implies that,
say, the dimension of the space of $(A_0,A_1,B_0,p_1,q_0)$ subject to
$A_1B_0-B_0A_0+p_1q_0=0$ is at most
$\frac{d_0(d_0-1)}{2}+d_1^2-d_1+\min(d_0,d_1)+\max(d_0,d_1)$. Summing up
the similar estimates over $0\leq l<N/2$ we obtain the desired inequality
$\dim\Upsilon^{-1}(0,\ldots,0)\leq\dim G_{-\varepsilon}(V_\bullet)+
\sum_{l\in I}d_l$. The opposite inequality follows from the computation of
the generic fiber of $\Upsilon$, and the proposition follows.
\qed

The following corollary is proved the same way as~\cite[Corollary~2.12]{fr}.
\begin{cor}
\label{ci}
$\sM^\varepsilon_{\ul{d}}$ is an irreducible reduced complete intersection
in $M^\varepsilon_{\ul{d}}$. \qed
\end{cor}

\begin{thm}
\label{normal}
$\fZ^\varepsilon_{\ul{d}}$ is a reduced irreducible normal scheme.
\end{thm}

\proof The same argument as in the proof of~\cite[Theorem~2.7.a)]{fr}.
We just list the minor changes necessary in the quadratic case.
As in {\em loc. cit.} we have to check the normality of the open subscheme
$U\subset\fZ^\varepsilon_{\ul{d}}$ defined as the preimage under the
factorization morphism $\Phi:\ \fZ^\varepsilon_{\ul{d}}\to\BA^{\ul{d}}$ of
an open subset ${\hat U}\subset\BA^{\ul{d}}$ formed by all the colored
configurations where at most 2 points collide. As in {\em loc. cit.} this
reduces to a few basic checks we already performed in the Examples:~\ref{c1}
when two points of the same color 0 or $N/2$ ({\em outmost color} for short)
collide; Example~\cite[2.8.1]{fr} when two points of the same color
$l,\ 0<l<N/2$ ({\em innermost color} for short) collide; Example~\ref{c2}
when a point of an outmost color collides with a point of an inmost color;
Example~\ref{tc1} when two points of different outmost colors collide;
Example~\cite[2.8.2]{fr} when two points of different innermost colors collide.
This completes the proof of the theorem.
\qed

\subsection{Symplectic zastava}
\label{sz}
For $\ol{d}=(d_0,\ldots,d_{N/2})\in\BN^I$ we consider the affine zastava space for symplectic
Lie group $G=Sp_N$ introduced in~\cite{bfg} and denoted
${\mathfrak U}^{\ol{d}}_{G;B}$. In the present paper we will denote it
$Z_\varepsilon^{\ol{d}}$. It is a reduced irreducible affine scheme containing
as an open subscheme the (smooth) moduli space
$\overset{\circ}{Z}{}_\varepsilon^{\ol{d}}$
of degree $\ol{d}$ based maps from $\BP^1$ to the affine flag scheme of
$G=Sp_N$.

Recall that $Sp_N$ is the fixed point subgroup of the involutive pinning-preserving outer
automorphism $\sigma:\ SL_N\to SL_N$. This automorphism acts on the affine
flag scheme of $SL_N$, and on the zastava spaces $Z^{\ul{d}}$
for $SL_N$. More precisely, we define
$\ul{d}=(\widetilde{d}_0,\ldots,\widetilde{d}_{N-1})\in\BN^N$
as follows: for $0\leq l\leq N/2$ we have $\widetilde{d}_l=d_l$, and for
$0<l<N/2$ we set $\widetilde{d}_{N-l}=d_l$. Then $\sigma$ acts on
$Z^{\ul{d}}$, and the fixed point subscheme (with the reduced
closed subscheme structure) is isomorphic to $Z_\varepsilon^{\ol{d}}$.
In other words, $Z_\varepsilon^{\ol{d}}$ is the closure in $Z^{\ul{d}}$
of $\overset{\circ}{Z}{}_\varepsilon^{\ol{d}}\cong
(\overset{\circ}{Z}{}^{\ul{d}})^\sigma$.

Recall the chainsaw quiver variety $\fZ_{\ul{d}}$ introduced
in~\cite{fr}. Theorem~\cite[2.7.b)]{fr} constructs a morphism
$\eta:\ \fZ_{\ul{d}}\to Z^{\ul{d}}$,
and~\cite[Theorem~3.5]{bf} proves that $\eta$ is an isomorphism. We will
identify $\fZ_{\ul{d}}$ and $Z^{\ul{d}}$ via $\eta$.
Under this identification the open subscheme
$\overset{\circ}{Z}{}^{\ul{d}}$ corresponds to the open subscheme
$\overset{\circ}{\fZ}{}_{\ul{d}}$ formed by the closed orbits of
stable and costable quadruples $(A_\bullet,B_\bullet,p_\bullet,q_\bullet)$, i.e.
such that $V_\bullet$ is generated by the action of $A_\bullet,B_\bullet$ from
the image of $p_\bullet$, and also $V_\bullet$ contains no nonzero subspaces
in $\on{Ker}q_\bullet$ closed with respect to $A_\bullet,B_\bullet$.
The fixed point subscheme $(\overset{\circ}{\fZ}{}_{\ul{d}})^\sigma$
coincides with $\overset{\circ}{\fZ}{}_{\ul{d}}^\varepsilon$: the open
subscheme of $\fZ_{\ul{d}}^\varepsilon$ formed by the closed orbits of
stable (equivalently, costable) quadruples
$(A_\bullet,B_\bullet,p_\bullet,q_\bullet)\in\sM_{\ul{d}}^\varepsilon$,
cf.~\cite[Table~1 and~Proposition~3.3]{bs}.

\begin{lem}
\label{miss}
The closed embedding $\overset{\circ}{\fZ}{}_{\ul{d}}^\varepsilon\cong
(\overset{\circ}{\fZ}{}_{\ul{d}})^\sigma\hookrightarrow
\overset{\circ}{\fZ}{}_{\ul{d}}$ extends to the closed embedding
$\fZ_{\ul{d}}^\varepsilon\hookrightarrow\fZ_{\ul{d}}$.
\end{lem}

\proof It suffices to check that any $G_{-\varepsilon}(V_\bullet)$-invariant
function on $\sM^\varepsilon_{\ul{d}}$ extends to a
$\prod_{l=0}^{N-1}GL(V_l)$-invariant function on $\sM_{\ul{d}}$.
This is immediately seen on the generators provided by the classical Invariant
Theory.
\qed

Now since $\overset{\circ}{\fZ}{}_{\ul{d}}^\varepsilon$
is dense in $\fZ_{\ul{d}}^\varepsilon=(\fZ_{\ul{d}})^\sigma$ we conclude
that $\fZ_{\ul{d}}^\varepsilon\subset\fZ_{\ul{d}}=Z^{\ul{d}}$
coincides with the closure of
$(\overset{\circ}{\fZ}{}_{\ul{d}})^\sigma$ in
$\fZ_{\ul{d}}=Z^{\ul{d}}$ . Since the symplectic
zastava scheme $Z_\varepsilon^{\ol{d}}$ also coincides with this closure,
we arrive at the following

\begin{thm}
\label{coinc}
There is a canonical isomorphism $\eta:\ \fZ_{\ul{d}}^\varepsilon\iso
Z^{\ol{d}}_\varepsilon$ making the following diagram commutative:
$$
\begin{CD}
\fZ_{\ul{d}}^\varepsilon  @>>> \fZ_{\ul{d}}\\
@V\eta VV  @VV\eta V\\
Z^{\ol{d}}_\varepsilon @>>> Z_{\ul{d}}
\end{CD}
$$
where the horizontal morphisms are the closed embeddings of the $\sigma$-fixed
points subschemes. \qed
\end{thm}

\begin{cor}
The symplectic zastava scheme $Z^{\ol{d}}_\varepsilon$ is normal. \qed
\end{cor}

\section{Hamiltonian reduction}

\subsection{Poisson structures}
According to~\cite{fkmm} (cf.~\cite[3.1,3.3]{fr}), the smooth scheme
$\overset{\circ}{Z}{}^{\ul{d}}$ carries a canonical symplectic
structure which extends as a Poisson structure to
$\fZ_{\ul{d}}=Z^{\ul{d}}$. This Poisson structure was
constructed in~\cite{fr} via Hamiltonian reduction. The restriction of
the symplectic form on $\overset{\circ}{Z}{}^{\ul{d}}$ to
$\overset{\circ}{Z}{}_\varepsilon^{\ol{d}}\cong
(\overset{\circ}{Z}{}^{\ul{d}})^\sigma$ coincides with the
canonical symplectic form~\cite{fkmm} on
$\overset{\circ}{Z}{}_\varepsilon^{\ol{d}}$. We conclude that the canonical
symplectic structure on $\overset{\circ}{Z}{}_\varepsilon^{\ol{d}}$ extends
as a Poisson structure to $Z_\varepsilon^{\ol{d}}\cong\fZ^\varepsilon_{\ul{d}}$,
and the $\sigma$-fixed point embedding $\fZ^\varepsilon_{\ul{d}}\hookrightarrow
\fZ_{\ul{d}}$ is Poisson. In the next subsection we will construct
this Poisson structure on $\fZ^\varepsilon_{\ul{d}}$ via Hamiltonian reduction.

\subsection{Classical} Recall the Hamiltonian reduction definition of the Poisson bracket on Zastava spaces in type A (see \cite{fr}). We ``triangulate'' the chainsaw quiver in the following way:

$$
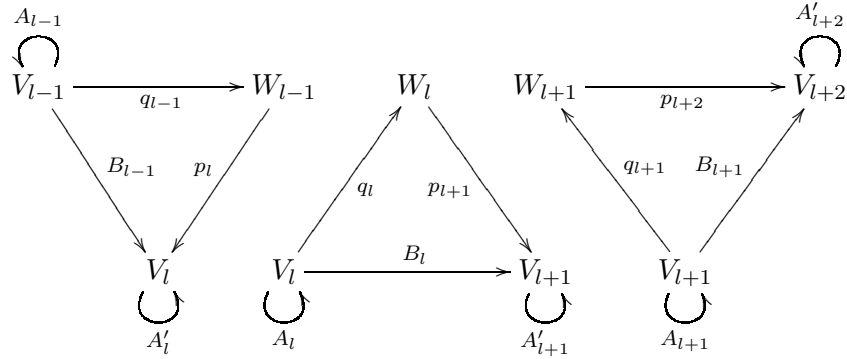
\xymatrix{
V_{l-1} \ar@(ur,ul)[]_{A_{l-1}} \ar[rdd]^{B_{l-1}} \ar[rr]_{q_{l-1}} &
& W_{l-1} \ar[ldd]_{p_l} & W_l \ar[rdd]_{p_{l+1}} &
W_{l+1} \ar[rr]_{p_{l+2}} && V_{l+2} \ar@(ur,ul)[]_{A'_{l+2}}       \\
&&&&&&\\
& V_l \ar@(dl,dr)[]_{A'_l} &
V_l \ar@(dl,dr)[]_{A_l} \ar[rr]^{B_l} \ar[ruu]_{q_l} &
& V_{l+1} \ar@(dl,dr)[]_{A'_{l+1}}
& V_{l+1} \ar@(dl,dr)[]_{A_{l+1}} \ar[ruu]^{B_{l+1}} \ar[luu]_{q_{l+1}} &
}$$

For a pair of vector spaces $U,V$ define the following $2$-step nilpotent Lie algebra: $$\fn(U,V):=U\oplus V^*\oplus (U\otimes V^*),$$ where the space $U\otimes V^*$ is central, $[U,U]=[V^*,V^*]=0$, and for $u\in U,\ v^\vee\in V^*$ one has $[u,v^\vee]=u\otimes v^\vee$.

To define the Poisson structure, we attach to each triangle of our graph the following Lie algebra
$$\fa_l:=(\fgl(V_{l})\oplus \fgl(V_{l+1}))\rtimes \fn(V_{l},V_{l+1})
$$ (the semidirect sum is with respect to the tautological action of $\fgl(V_l)$ on $V_l$ and $\fgl(V_{l+1})$ on $V_{l+1}^*$).

Consider the Lie algebra
$$\fa_{\ul{d}}:=\bigoplus\limits_{l\in\BZ/N\BZ}\fa_l=\bigoplus\limits_{l\in\BZ/N\BZ}(\fgl(V_l)\oplus \fgl(V_{l+1}))\rtimes \fn(V_l,V_{l+1})
$$

The coadjoint representation of $\fa_{\ul{d}}$ is the space $\fa_{\ul{d}}^*=\{(A_l,A_l',B_l,p_l,q_l)_{l\in\BZ/N\BZ}\}$, where
$$A_l\in\End(V_l),\quad A_l'\in\End(V_l),\quad B_l\in\Hom(V_l,V_{l+1}),\quad p_l\in V_l,\quad q_l\in V_l^*.$$

Consider the subvariety $S_{\ul{d}}\subset\fa_{\ul{d}}^*$ defined by the following equations:
\begin{equation}\label{poisson-ideal}
B_lA_l+A_{l+1}'B_l+p_{l+1}q_l=0,\quad l\in\BZ/N\BZ.
\end{equation}

Let $\fgl(V_l)_{\diag}$ be the diagonal $\fgl(V_l)$ inside $\fgl(V_l)\oplus\fgl(V_l)\subset\fa_{\ul{d}}$ and $\pi:\fa_{\ul{d}}^*\to\fgl(V_l)^*_{\diag}$ be the projection. Then the Drinfeld Zastava space $\fZ_{\ul{d}}$ is identified with the Hamiltonian reduction $S_{\ul{d}}/\!\!/\!\!/\bigoplus\limits_{l\in\BZ/N\BZ}\fgl(V_l)_{\diag}=\pi^{-1}(0)\cap S_{\ul{d}}/\!\!/\prod\limits_{l\in\BZ/N\BZ}GL(V_l)_{\diag}$. This provides a natural Poisson bracket on $\fZ_{\ul{d}}$.

The involution $\sigma$ acts on the space $\fa_{\ul{d}}^*$ as follows:
$$
A_l\mapsto -A_{N-l}'^*,\ A_l'\mapsto -A_{N-l}^*,\ B_l\mapsto B_{N-l}^*,\ p_l\mapsto q_{N-l}^*,\ q_l\mapsto p_{N-l}^*.
$$

\begin{rem}
{\em Strictly speaking, $\sigma$ is not an involution on $\fa_{\ul{d}}^*$ since $p^{**}=-p$ and $q^{**}=-q$, but becomes an involution after the Hamiltonian reduction.}
\end{rem}

To describe the fixed point set, we consider the half of the chainsaw quiver formed by the vertices $l\in I$. Define the Lie algebra
$$\fa_{\ul{d}}^\varepsilon:=\bigoplus\limits_{l=0}^{\frac{N}{2}-1}\fa_l=\bigoplus\limits_{l=1}^{\frac{N}{2}-1}(\fgl(V_{l})\oplus \fgl(V_{l+1}))\rtimes \fn(V_{l},V_{l+1}).
$$

The coadjoint representation of $\fa_{\ul{d}}^\varepsilon$ is the space
$\fa_{\ul{d}}^{\varepsilon*}=\{(A_l,A_l',B_l,p_l,q_l)_{l\in I_0}, A_0,q_0,B_0,A_{\frac{N}{2}}',p_{\frac{N}{2}} \}$, where
$$A_l\in\End(V_l),\quad A_l'\in\End(V_l),\quad B_l\in\Hom(V_l,V_{l+1}),\quad p_l\in V_l,\quad q_l\in V_l^*.$$

The invariant subvariety $S_{\ul{d}}^\varepsilon\subset\fa_{\ul{d}}^{\varepsilon*}$ is again defined by the equations~(\ref{poisson-ideal}):
$$
B_{l}A_{l}+A_{l+1}'B_{l}+p_{l+1}q_{l}=0,\quad l=0,\ldots,\frac{N}{2}-1.
$$

Let  $\fo(V_l)\subset\fgl(V_l)\subset\fa_{\ul{d}}^\varepsilon$ for $l\in I_1$ be the orthogonal Lie subalgebra and let $\pi:\fa_{\ul{d}}^{\varepsilon*}\to\bigoplus\limits_{l\in I_0}\fgl(V_l)^*_{\diag}\oplus\bigoplus\limits_{l\in I_1}\fo(V_l)^*$ be the projection. Then the symplectic Drinfeld Zastava space $\fZ_{\ul{d}}^\varepsilon$ is identified (as a Poisson variety) with the Hamiltonian reduction: $$
S^\varepsilon_{\ul{d}}\slash\!\!\slash\!\!\slash\bigoplus\limits_{l\in I_0}\fgl(V_l)_{\diag}\oplus\bigoplus\limits_{l\in I_1}\fo(V_l)=(\pi^{-1}(0)\cap S_{\ul{d}}^\varepsilon)
\slash\!\!\slash\prod\limits_{l\in I_0}GL(V_l)_{\diag}\times\prod\limits_{l\in I_1}O(V_l).
$$
We denote the group $\prod\limits_{l\in I_0}GL(V_l)_{\diag}\times\prod\limits_{l\in I_1}O(V_l)$ simply by $G_{\ul{d}}$, and the corresponding Lie algebra $\bigoplus\limits_{l\in I_0}\fgl(V_l)_{\diag}\oplus\bigoplus\limits_{l\in I_1}\fo(V_l)$ by $\fg_{\ul{d}}$.

\subsection{Quantum} The natural quantization of the coordinate ring $\BC[\fa_{\ul{d}}^{\varepsilon*}]$ is the enveloping algebra $U(\fa_{\ul{d}}^\varepsilon)$. It will be convenient to gather the generators of $U(\fa_{\ul{d}}^\varepsilon)$ (i.e. the basis elements of the Lie algebra $\fa_{\ul{d}}^\varepsilon$) into the following $U(\fa_{\ul{d}}^\varepsilon)$-valued matrices:

$$
A_k, B_k, q_k, A'_l, p_l,\ 0\le k < \frac{N}{2},\ 0< l \le \frac{N}{2}.
$$

According to \cite{fr} the coefficients of the following matrices form a subspace $R\subset U(\fa_{\ul{d}}^\varepsilon)$ invariant with respect to the adjoint action:

\begin{equation}\label{2sided-ideal}
B_lA_l+A'_{l+1}B_l+p_{l+1}q_l,\quad l=0,\ldots,\frac{N}{2}-1,\ i=1,\ldots,d_{l+1},\ j=1,\ldots,d_{l}.
\end{equation}

Equivalently, $U(\fa_{\ul{d}}^\varepsilon)R$ is a two-sided ideal in $U(\fa_{\ul{d}}^\varepsilon)$).

The natural quantization of the coordinate ring of the space $\fZ_{\ul{d}}^\varepsilon$ is the \emph{quantum Hamiltonian reduction} $\YO^{\varepsilon}_{\ul{d}}:=\left(U(\fa_{\ul{d}}^\varepsilon)/U(\fa_{\ul{d}}^\varepsilon)(R+\fg_{\ul{d}})\right)^{G_{\ul{d}}}$. The ring $\YO_{\ul{d}}^\varepsilon$ has a natural filtration coming from the PBW filtration on $U(\fa_{\ul{d}}^\varepsilon)$.

\begin{prop}[PBW property]\label{PBW} We have $\gr\ \YO^{\varepsilon}_{\ul{d}}=\BC[\fZ_{\ul{d}}^\varepsilon]$.
\end{prop}

\begin{proof}
The proof is a word-to-word repetition of that of Proposition~3.28 from \cite{fr}.
\end{proof}

We consider the following elements of $\YO^{\varepsilon}_{\ul{d}}$:
$$a_{l,r}:=\Tr A_l^r,\
r=1,2,\ldots,\ l\in I ;$$
$$b_{l,s}:=q_lA_l^sp_l,\ s=0,1,\ldots,\ l\in I.$$

We also introduce the following elements:

\begin{equation}
b_{k,l;s_k,\ldots,s_l}:=q_l A_l^{s_l}B_{l-1}A_{l-1}^{s_{l-1}}B_{l-2}\ldots B_kA_k^{s_k}p_k,\quad k\le l\in\BZ,\ s_i\in\BZ_{\ge0}.
\end{equation}

\begin{equation}
c_{k,l;s_k,\ldots,s_l}:=B_l A_l^{s_l}B_{l-1}A_{l-1}^{s_{l-1}}B_{l-2}\ldots B_kA_k^{s_k},\quad l=k+mN,\ s_i\in\BZ_{\ge0}.
\end{equation}

From the definitions we get the following relations:

\begin{lem}\label{bb-commutators} Let $k < l+1$. Then $$
[b_{k,l;0,\ldots,0},b_{l+1,0}]=\left\{ \begin{array}{ll}b_{k,l+1;0,\ldots,0,0}&l\pm k\ne 0,2\mod N,\ 2l+2\ne0\mod N\\
2b_{k,l+1;0,\ldots,0,0}&2l+2=0\mod N\\
b_{k,l+1;0,\ldots,0,0}-b_{k-1,l;0,\ldots,0,0}&l+k=0\mod N\\
2(b_{k,l+1;0,\ldots,0,0}-b_{k-1,l;0,\ldots,0,0})&l+k=0\mod N\ \text{and}\ 2l+2=0\mod N\\
b_{k,l+1;0,\ldots,0,0}-b_{k-1,l;0,\ldots,0,0}&l-k=2\mod N\end{array}\right.$$
\end{lem}

\begin{lem}\label{ab-commutators} For $k\le m\le l$, we have $[a_{m,r},b_{k,l;s_k,\ldots,s_l}]=\lambda b_{k,l;s_k,\ldots,s_m+r-1\ldots,s_l} + L$, where $\lambda\in\BC\backslash\{0\}$, $L\in\YO^{\varepsilon}_{\ul{d}}$ is expressed in $b_{k',l';s_{k'},\ldots,s_{l'}}$ with $l'-k'<l-k$, and $\deg L\le\deg b_{k,l;s_k,\ldots,s_m+r-1\ldots,s_l}$ with respect to the PBW filtration.
\end{lem}

\begin{proof}
Straightforward.
\end{proof}

\begin{lem}\label{C1}\emph{($C_1$ case)} Let $p_0,p_1,A_0,A_1,B$ be the (matrices of) generators of the algebra $\fa_{\ul d}$ for $N=2$. Then the algebra $\YO^{\varepsilon}_{\ul{d}}$ is generated by $a_{l,r}:=\Tr A_l^r,\ b_{l,0}:=p_l^*p_l$ with $l=0,1,\ r=1,\ldots, d_l$.
\end{lem}

\begin{proof}
According to Lemma~\ref{ab-commutators}, it suffices to check that the invariants of the form $p^*_0(B^*B)^mp_0$, $p^*_1B(B^*B)^mp_0$, $p^*_0(B^*B)^mB^*p_1$, $p^*_1(BB^*)^mp_0$ and $\Tr (BB^*)^m$ can be expressed in $a_{l,r}, b_{l,0}$. This is easily checked by induction on $m$.
\end{proof}

\begin{prop}\label{quantum-generators-affine} The algebra $\YO^{\varepsilon}_{\ul{d}}$ is generated by $a_{l,r},b_{l,s}$ with $l\in I,\ r=1,\ldots, d_l,\ s=0,\ldots,d_l-1$.
\end{prop}

\begin{proof}
Arguing in the same way as in Proposition~3.35 of \cite{fr} we reduce the problem to expressing $b_{k,l;0,0,\ldots,0}$ and $c_{0,mN;0,\ldots,0}$ via $a_{l,r},b_{l,s}$.

By Lemma~\ref{bb-commutators} for $l-k<N-1$ we have $b_{k,l;0,0,\ldots,0}=\lambda[[\ldots[b_{k,0}b_{k+1,0}]\ldots,b_{l-1,0}],b_{l,0}]$, where $\lambda$ is a nonzero number.
Thus $b_{k,l;0,0,\ldots,0}$ with $l-k<n$ are expressed via $a_{l,r},b_{l,s}$.  Suppose that $b_{k,k+mN-1;0,0,\ldots,0}$ for some $m\in\BZ_+$ is expressed via $a_{l,r},b_{l,s}$, then for $l-k<N$ we have $b_{k,l+mN;0,0,\ldots,0}=\lambda[[\ldots[b_{k,k+mN-1;0,0,\ldots,0}b_{k+N+1,0}]\ldots,b_{l+N-1,0}],b_{l+N,0}]$. Thus $b_{k,l+mN;0,0,\ldots,0}$ with $l-k<N$ are expressed via $a_{l,r},b_{l,s}$ as well. So, the problem reduces to expressing $b_{k,k+mN-1;0,\ldots,0,0}$ and $c_{0,mN;0,\ldots,0}$ via $a_{l,r},b_{l,s}$.

Let $\ul{D}=(d_0,d_{\frac{N}{2}})$. Define the homomorphism $\Phi:U(\fa_{\ul{D}}^\varepsilon)\to U(\fa_{\ul{d}}^\varepsilon)$ as $$\Phi(A_0)=A_0,\ \Phi(A_1)=A_{\frac{N}{2}},\ \Phi(B)=B_{\frac{N}{2}-1}\cdot B_{\frac{N}{2}-2}\cdot\ldots \cdot B_0,\ \Phi(p_0)=p_0,\ \Phi(p_1)=B_{\frac{N}{2}-1}\cdot\ldots \cdot B_1p_1.$$ Note that $\Phi(\YO^{\varepsilon}_{\ul{D}})\subset\YO^{\varepsilon}_{\ul{d}}$. By Lemma~\ref{C1}, $\YO^{\varepsilon}_{\ul{D}}$ is generated by the elements $a_{l,r}:=\Tr A_l^r,\ b_{l,0}:=p_l^*p_l$ with $l=0,1,\ r=1,\ldots, d_l$. We have $\Phi(a_{0,r})=a_{0,r}$, $\Phi(a_{1,r})=a_{\frac{N}{2},r}$, $\Phi(b_{0,0})=b_{0,0}$, $\Phi(b_{1,0})=b_{1,N-1;0,\ldots,0}$. Thus everything from $\Phi(\YO^{\varepsilon}_{\ul{D}})$ is expressed via $a_{l,r},b_{l,s}$. On the other hand, $\Phi(b_{0,2m-1;0,\ldots,0})=b_{0,mN-1;0,\ldots,0}$ and $\Phi(c_{0,2m;0,\ldots,0})=c_{0,mN;0,\ldots,0}$.
\end{proof}

\section{Yangians}

\subsection{Yangian of $\fsp_{N}$}
\label{yang fin}
Let $(c_{kl})_{k,l=1,2,\ldots,\frac{N}{2}}$ stand for the symmetrized Cartan matrix of
$\fsp_{N}$. That is
$c_{kk}=4$ for $k=\frac{N}{2};\ c_{kk}=2$ for $0<k<\frac{N}{2};\
c_{kl}=0$ for $|k-l|>1;\ c_{kl}=-1$ for $0<k,l<\frac{N}{2}$ and $l=k\pm1;\
c_{kl}=-2$ otherwise.

The Yangian $Y(\fsp_{N})$ is generated
by $\bx_{k,r}^\pm,\bh_{k,r},\ k=1,2,\ldots,\frac{N}{2},\ r\in\BN$, with the following
relations:

\begin{equation}
\label{11}
[\bh_{k,r},\bh_{l,s}]=0,\ [\bh_{k,0},\bx_{l,s}^\pm]=\pm c_{kl}\bx_{l,s}^\pm,
\end{equation}

\begin{equation}
\label{12}
2[\bh_{k,r+1},\bx_{l,s}^\pm]-2[\bh_{k,r},\bx_{l,s+1}^\pm]=
\pm c_{kl}(\bh_{k,r}\bx_{l,s}^\pm+\bx_{l,s}^\pm\bh_{k,r}),
\end{equation}

\begin{equation}
\label{13}
[\bx^+_{k,r},\bx^-_{l,s}]=\delta_{kl}\bh_{k,r+s},
\end{equation}

\begin{equation}
\label{14}
2[\bx_{k,r+1}^\pm,\bx_{l,s}^\pm]-2[\bx_{k,r}^\pm,\bx_{l,s+1}^\pm]=
\pm c_{kl}(\bx_{k,r}^\pm\bx_{l,s}^\pm+\bx_{l,s}^\pm\bx_{k,r}^\pm),
\end{equation}

\begin{equation}
\label{15}
[\bx_{k,r}^\pm,[\bx_{k,p}^\pm,\bx_{l,s}^\pm]]+
[\bx_{k,p}^\pm,[\bx_{k,r}^\pm,\bx_{l,s}^\pm]]=0,\
k=l\pm1,\ k\in I, l\in I_0,\ \forall p,r,s\in\BN.
\end{equation}

\begin{equation}
\label{16}
\sum\limits_{\sigma\in S_3}[\bx_{k,r_{\sigma(3)}},[\bx_{k,r_{\sigma(2)}},[\bx_{k,r_{\sigma(1)}},\bx_{l,s}]]]=0,\
k=l\pm1,\ k\in I, l\in I_1,\ \forall r_1,r_2,r_3,s\in\BN.
\end{equation}

We will consider the ``Borel subalgebra'' $\YO^{\varepsilon}$ of the Yangian, generated by $\bx_{k,r}^+$ and $\bh_{k,r}$.
For a formal variable $u$ we introduce the generating series
$\bh_k(u):=1+\sum_{r=0}^\infty\bh_{k,r}u^{-r-1};\
\bx_k^+(u):=\sum_{r=0}^\infty\bx_{k,r}^+ u^{-r-1}$.

We also consider a bigger algebra $\DO\YO^{\varepsilon}$, the ``Borel subalgebra of the Yangian double'', generated by all Fourier components of the series $\bh_k(u):=1+\sum_{r=0}^\infty\bh_{k,r}u^{-r-1};\
\bx_k^+(u):=\sum_{r=-\infty}^\infty\bx_{k,r}^+ u^{-r-1}$ (i.e. the generating series $\bx_k^+(u)$ are infinite in both positive and negative directions) with the defining relations~(\ref{11},\ref{12},\ref{14},\ref{15},\ref{16}). The algebra $\YO^{\varepsilon}$ is then the subalgebra generated by negative Fourier components of $\bx_k^+(u)$ and $\bh_k(u)$ due to PBW property of the Yangians. We can then rewrite the equations~(\ref{12},\ref{14}) in the following form

\begin{equation}
\label{12'}
\bh_k(u)\bx_l^+(v)\frac{2u-2v-c_{kl}}{2u-2v+ c_{kl}}=\bx_l^+(v)\bh_k(u).
\end{equation}

\begin{equation}
\label{14'}
\bx_k^+(u)\bx_l^+(v)(2u-2v- c_{kl})=(2u-2v+ c_{kl})\bx_l^+(v)\bx_k^+(u).
\end{equation}

The function $\frac{2u-2v-c_{kl}}{2u-2v+ c_{kl}}$ here is understood as a formal power series in $u^{-1},\ v^{-1},\ u^{-1}v$, hence the equation~(\ref{12'}) is well-defined.

Following \cite{fr},
we will use a little bit different generators of the Cartan subalgebra of the Yangian,
\begin{equation}
\bA_k(u):=u^{d_k}+A_{k,0}u^{d_k-1}+\ldots+A_{k,r}u^{d_k-r-1}+\ldots,
\end{equation}
obtained as the (unique) solution of the system of functional equations:

\begin{equation}
\bh_k(u)=\bA_k(u+\frac{1}{2})^{-1}\bA_k(u-\frac{1}{2})^{-1}\bA_{k-1}(u)\bA_{k+1}(u)(u+\frac{1}{2})^{d_k}(u-\frac{1}{2})^{d_k}u^{-d_{k-1}}u^{-d_{k+1}},
\end{equation}
for $k=1,2,\ldots,\frac{N}{2}-1$, and

\begin{equation}
\bh_{\frac{N}{2}}(u)=\bA_{\frac{N}{2}}(u+1)^{-1}\bA_{\frac{N}{2}}(u-1)^{-1}\bA_{\frac{N}{2}-1}(u)\bA_{\frac{N}{2}-1}(u+\frac{1}{2})(u+1)^{d_{\frac{N}{2}}}(u-1)^{d_{\frac{N}{2}}}u^{-d_{\frac{N}{2}-1}}(u+\frac{1}{2})^{-d_{\frac{N}{2}-1}}.
\end{equation}

Here we take $\bA_0(u)=1$

\begin{lem}\label{A-generators}The generators $\bA_k(u)$ of $\DO\YO^{\varepsilon}$ satisfy the relations
\begin{equation}\label{a-rel}
\bA_k(u)\bx_l^+(v)\frac{2u-2v+\frac{c_{kk}\delta_{kl}}{2}}{2u-2v- \frac{c_{kk}\delta_{kl}}{2}}=\bx_l^+(v)\bA_k(u).
\end{equation}
\end{lem}

\begin{lem}\label{yangian-rel}
Let $\bA_k(u)$ and $\bx_l^+(u)$ be the generating series of $\DO\YO^{\varepsilon}$. Then the series $$
\ba_k(u)=\frac{\bA_k(u-\frac{c_{kk}}{4})}{\bA_k(u+\frac{c_{kk}}{4})}=1-d_ku^{-1}-\sum\limits_{r=1}^\infty\ba_{k,r}u^{-r-1},\quad \bx_l^+(u)
$$ satisfies the following commutator relations
\begin{equation}\label{a-x-defnrel}
[\ba_k(u),\bx_l^+(v)](u-v)=-\frac{\frac{c_{kk}^2}{4}\delta_{kl}}{u-v}\bx_l^+(v)\ba_k(u),\quad [\ba_k(u),\ba_l(v)]=0.
\end{equation}
The series $\ba_k(u),\ \bx_l^+(u)$ generate $\DO\YO^{\varepsilon}$ with the defining relations~(\ref{a-x-defnrel}),~(\ref{14})~and~(\ref{15}), and their negative Fourier components generate $\YO^\varepsilon$.
\end{lem}

\begin{proof} For $k\ne l$ the relation is obvious, for $k=l$ we have
$$
\ba_k(u)\bx_k^+(v)\frac{u-\frac{c_{kk}}{4}-v+\frac{c_{kk}}{4}}{u-\frac{c_{kk}}{4}-v-\frac{c_{kk}}{4}}\cdot\frac{u+\frac{c_{kk}}{4}-v-\frac{c_{kk}}{4}}{u+\frac{c_{kk}}{4}-v+\frac{c_{kk}}{4}}=\bx_k^+(v)\ba_k(u).
$$therefore
$$
\ba_k(u)\bx_k^+(v)\frac{(u-v)^2}{(u-v)^2-\frac{c_{kk}^2}{4}}=\bx_k^+(v)\ba_k(u).
$$

One can inductively express $\bA_{k,r}$ via $\ba_{k,s}$ with $s\le r+1$, hence $\DO\YO^{\varepsilon}$ is generated by $\ba_k(u)$ and $\bx_l^+(u)$. On the other hand, the quotient of $\BC[\ba_{k,r}]_{r=1}^{\infty}\cdot\DO\YO^+$ by the relation~(\ref{a-x-defnrel}) is $\BC[\ba_{k,r}]_{r=1}^{\infty}\otimes\DO\YO^+$ as a filtered vector space. The same argumentation for $\YO^{\varepsilon}$. Hence the assertion.
\end{proof}

\subsection{Yangian of $\widehat\fsp_{N}$}
Let $(c_{kl})_{k,l\in I}$ stand for the symmetrized Cartan matrix of
$\widehat\fsp_{N}$. That is
$c_{kk}=4$ for $k=0$ or $k=\frac{N}{2};\ c_{kk}=2$ for $0<k<\frac{N}{2};\
c_{kl}=0$ for $|k-l|>1;\ c_{kl}=-1$ for $0<k,l<\frac{N}{2}$ and $l=k\pm1;\
c_{kl}=-2$ otherwise.

As for the finite case, we will consider the ``affine Borel Yangian''. This is an associative algebra $\widehat{\YO}{}^\varepsilon$ generated by the series
\begin{equation}
\bx_k^+(u):==1+\sum\limits_{r=0}^{\infty}\bx_{k,r}u^{-r-1},
\end{equation}
\begin{equation}
\bA_k(u):=u^{d_k}+\sum\limits_{r=0}^{\infty}\bA_{k,r}u^{d_k-r-1},
\end{equation}
with $k\in\BZ$ subject to the relations

\begin{equation}
\bA_k(u)\bA_l(v)=\bA_l(v)\bA_k(u),
\end{equation}
\begin{equation}
\bx_k^\pm(u)\bx_l^\pm(v)(2u-2v\mp c_{kl})=\bx_l^\pm(v)\bx_k^\pm(u)(2u-2v\pm c_{kl}),
\end{equation}
where $(c_{kl})$ stands for the symmetrized Cartan matrix of $\wt{C_n}$;
\begin{equation}
\bA_k(u)\bx_l^+(v)\frac{2u-2v+\frac{c_{kl}\delta_{kl}}{2}}{2u-2v-\frac{c_{kl}\delta_{kl}}{2}}=\bx_l^+(v)\bA_k(u),
\end{equation}
in the sense that negative Fourier components of LHS and RHS are equal, and the Serre relations (\ref{15}) and (\ref{16}).

\subsection{Symplectic Yangian and symplectic Zastava spaces}

\begin{thm}\label{yangian-quotient-general}
The algebra $\YO^{\varepsilon}_{\ul{d}}$ is a quotient of the Borel Yangian $\widehat{\YO}{}^\varepsilon$ of
$\widehat\fsp_N$ by some ideal containing $\bA_{k,r}=0\ \text{for}\ r>d_k$.
\end{thm}
\begin{proof}
For $k\in I,\ l\in I_0$, introduce the following generating series in $\YO^{\varepsilon}_{\ul{d}}$:
\begin{equation}
a_k(u):=1-d_ku^{-1}-\sum\limits_{r=1}^\infty a_{k,r}u^{-r-1}, \quad b_l(u):=\sum\limits_{s=0}^\infty b_{l,s}u^{-s-1}
\end{equation}

For $l\in I_1, i=1,2,\ldots,d_l$, introduce the following generating series in $U(\fa_{\ul{d}}^{\varepsilon})$ (warning: not in $\YO^{\varepsilon}_{\ul{d}}$):
\begin{equation}
b_l^{(i)}(u):=\sum\limits_{s=0}^\infty b_{l,s}^{(i)}u^{-s-1},\quad b_{l,s}^{(i)}:= (A_l^sp_l)^{(i)},
\end{equation}
the $i$-th coordinate of the vector $A_l^sp_l$ in the orthonormal basis of $V_l$. The following relations hold (see \cite{fr}):

\begin{lem}\label{b-rel} The following relations hold:
\begin{equation} (u-v)[b_k(u),b_k(v)]=(b_k(u)b_k(v)+b_k(v)b_k(u))\ \text{for}\ k\in I_0,
\end{equation}
\begin{equation} 2(u-v)[b_k(u),b_l(v+d_l)]=-(b_k(u)b_l(v+d_l)+b_l(v+d_l)b_k(u))\ \text{for}\ k,l\in I_0,\ l=k+1,
\end{equation}
\begin{equation}
(u-v)[a_k(u),b_l(v)]=-\frac{\delta_{kl}}{u-v}b_l(v)a_k(u)\ \text{for}\ k\in I,\ l\in I_0.
\end{equation}
\begin{equation} 2(u-v)[b_k^{(i)}(u),b_k^{(j)}(v)]=c_{kl}(b_k^{(i)}(u)b_k^{(j)}(v)+b_k^{(j)}(v)b_k^{(i)}(u))\ \text{for}\ k\in I_1,\ i,j=1,\ldots,d_l,
\end{equation}
\begin{equation} 2(u-v)[b_k^{(i)}(u),b_l(v+d_l)]=c_{kl}(b_k^{(i)}(u)b_l(v+d_l)+b_l(v+d_l)b_k^{(i)}(u))\ \text{for}\ k\in I_1,\ l=k+1,\ i=1,\ldots,d_l,
\end{equation}
\begin{equation} 2(u-v)[b_k(u),b_l^{(i)}(v+d_l)]=c_{kl}(b_k(u)b_l^{(i)}(v+d_l)+b_l^{(i)}(v+d_l)b_k(u))\ \text{for}\ l\in I_1,\ l=k+1,\ i=1,\ldots,d_l,
\end{equation}
\begin{equation}
(u-v)[a_k(u),b_l^{(i)}(v)]=-\frac{\delta_{kl}}{u-v}b_l^{(i)}(v)a_k(u)\ \text{for}\ k\in I,\ l\in I_1,\ i=1,\ldots,d_l.
\end{equation}
\end{lem}

\begin{proof}
This follows from Propositions~3.24~and~3.32 of \cite{fr}.
\end{proof}

\begin{lem} We have
\begin{equation}
[b_{k,r_2},[b_{k,r_1},b_{l,s}]]+[b_{k,r_1},[b_{k,r_2},b_{l,s}]]=0\quad\text{for}\ k,l\in I_0,\ |k-l|=1.
\end{equation}
\begin{equation}
[b_{k,r_2},[b_{k,r_1},b_{l,s}^{(i)}]]+[b_{k,r_1},[b_{k,r_2},b_{l,s}^{(i)}]]=0\quad\text{for}\ k\in I,\ l\in I_1,\ |k-l|=1,\ i=1,\ldots,d_l,
\end{equation}
\begin{equation}
[b_{k,r_2}^{(i)},[b_{k,r_1}^{(j)},b_{l,s}]]+[b_{k,r_1}^{(j)},[b_{k,r_2}^{(i)},b_{l,s}]]=0\quad\text{for}\ k\in I_1,\ l\in I,\ |k-l|=1\ i,j=1,\ldots,d_k.
\end{equation}
\end{lem}

\begin{proof}
This follows from Proposition~3.32 of \cite{fr}.
\end{proof}

For $l\in I$, let $D_l(u)$ be the (unique) solution of the functional equation
\begin{equation}
a_l(u):=D_l(u-\frac{1}{2})D_l(u+\frac{1}{2})^{-1}.
\end{equation}

We have $$
D_l(u)b_k(v)\frac{2u-2v+\delta_{kl}}{2u-2v-\delta_{kl}}=b_k(v)D_l(u)\ \text{for}\ k\in I_0,$$
and
$$
D_l(u)b_k^{(i)}(v)\frac{2u-2v+\delta_{kl}}{2u-2v-\delta_{kl}}=b_k^{(i)}(v)D_k(u)\ \text{for}\ k\in I_1.$$

Set $\wt{b_k^{(i)}}(u):=D_k(u-\frac{1}{2})^{-1}b_k^{(i)}(u)$. From Lemma~\ref{b-rel}, we have $$
\wt{b_k^{(i)}}(u)\wt{b_k^{(j)}}(v)=\wt{b_k^{(j)}}(v)\wt{b_k^{(i)}}(u).
$$
Note that the rest of the relations for $b_k^{(i)}$ from lemma~\ref{b-rel} remain the same for $\wt{b_k^{(i)}}$.

For $l\in I_1,\ i,j=1,\ldots,d_l$, set
\begin{equation}
\wt{b_l^{(ij)}}(u):=D(u-\frac{1}{2})\wt{b_l^{(i)}}(u)\wt{b_l^{(j)}}(u+1).
\end{equation}
Note that
\begin{equation}\label{btilde1}
\wt{b_l^{(ii)}}(u)\wt{b_l^{(jj)}}(v)\frac{u-v+2}{u-v-2}=\wt{b_l^{(jj)}}(v)\wt{b_l^{(ii)}}(u)
\end{equation} and
\begin{equation}\label{btilde2}
\wt{b_l^{(ii)}}(u)b_k(v)\frac{u-v+1}{u-v-1}=b_k(v)\wt{b_l^{(ii)}}(u)
\end{equation}
for $|k-l|=1$.

For $l\in I_1$ set $\wt{b_l}(u)=\sum\limits_{i=1}^{d_l}\wt{b_l^{(ii)}}(u)$. From Proposition~\ref{quantum-generators-affine}, we see that the algebra $\YO^{\varepsilon}_{\ul{d}}$ is generated by (Fourier coefficients of) $D_l(u)$, $b_k(u)$ for $k\in I_0$ and $\wt{b_k}(u)$ for $k\in I_1$. Now the Theorem reduces to the following
\begin{lem}
There is a homomorphism $\varphi_{\ul{d}}:\widehat{\YO}{}^{\varepsilon}\to\YO^{\varepsilon}_{\ul{d}}$ sending $\bA_k(u)$ to $D_k(u+\sum\limits_{m=1}^kd_m)$ and $\bx_l^+(u)$ to $b_l(u+\sum\limits_{m=1}^ld_m)$ for $l\in I_0$ and to $\wt{b_l}(u+\sum\limits_{m=1}^ld_m)$ for $i\in I_1$.
\end{lem}
\begin{proof}
We need to prove the quadratic and the Serre relations for the elements $D_l(u)$, $b_k(u)$. The quadratic relations follow from the relations~(\ref{btilde1})~and~(\ref{btilde2}). The proof of the Serre relations is entirely similar to that of Proposition~3.32 from \cite{fr}.
\end{proof}

According to the Newton identity (see Theorem~7.1.3 of \cite{mbook}), we have
\begin{equation}
a_l(u)=\frac{C_l(-u+d_l)}{C_l(-u+d_l-1)},
\end{equation}
where $C_l(u)$ is the Capelli determinant. This means that $D(u)=C(-u+d_l-\frac{1}{2})$. In particular, $D_{l,r}=0$ for $r>d_l$.
\end{proof}

\begin{conj} $\YO^{\varepsilon}_{\ul{d}}=\widehat{\YO}{}^{\varepsilon}/\{\bA_{k,r}\ |\ r>d_k\}$.
\end{conj}

\begin{acknowledgements}
We are grateful to D.~Panyushev and D.~Timashev
for the help with references.
\end{acknowledgements}

\end{document}